\let\CMSKernelLabel\label
\let\CMSKernelRefstepcounter\refstepcounter
\documentclass[final,oneeqnum,onefignum,onetabnum]{cmslatex}
\usepackage{cms-compat}

\usepackage[utf8]{inputenc}
\usepackage[T1]{fontenc}
\usepackage{lmodern,microtype}
\usepackage{amsmath,amssymb,amsthm,mathtools,mathrsfs,bm}
\usepackage{booktabs,tabularx,array,enumitem,graphicx}
\usepackage[numbers,sort&compress]{natbib}
\usepackage{xcolor}
\definecolor{cmslink}{RGB}{0,65,120}
\usepackage{hyperref}
\newcommand{\papertitle}{Curvature Bounds for the Cost of Coherence in Conditional Optimal Transport}
\hypersetup{colorlinks=true,allcolors=cmslink,
  pdftitle={\papertitle},
  pdfauthor={Xiaozhen WANG, Zihan Zhang, and Xiang Zhou}}
\usepackage[nameinlink,noabbrev]{cleveref}
\newcommand{\doi}[1]{\begingroup\urlstyle{same}\url{https://doi.org/#1}\endgroup}

\newtheorem{theorem}{Theorem}[section]
\newtheorem{proposition}[theorem]{Proposition}
\newtheorem{lemma}[theorem]{Lemma}
\newtheorem{corollary}[theorem]{Corollary}
\theoremstyle{definition}
\newtheorem{definition}[theorem]{Definition}
\newtheorem{assumption}{Assumption}[section]

\theoremstyle{remark}

\crefname{assumption}{Assumption}{Assumptions}
\Crefname{assumption}{Assumption}{Assumptions}

\makeatletter
\renewcommand{\eqref}[1]{\hyperref[#1]{\textup{\tagform@{\ref*{#1}}}}}
\makeatother

\newcommand{\X}{\mathcal X}
\newcommand{\V}{\mathcal V}
\newcommand{\Eloc}{\mathcal E_{\mathrm{loc}}}

\newcommand{\Dcoh}{\Delta_{\mathrm{coh}}}
\newcommand{\Aact}{\mathcal A}
\newcommand{\Qcoh}{\mathscr Q}
\newcommand{\mubar}{\overline\mu}
\newcommand{\dif}{\,\mathrm d}
\newcommand{\E}{\mathbb E}
\newcommand{\R}{\mathbb R}
\newcommand{\T}{\mathbb T}
\DeclareMathOperator{\diver}{div}
\DeclareMathOperator{\Fl}{Fl}

\setlist[enumerate]{itemsep=3pt,topsep=5pt}
\allowdisplaybreaks[1]

\title{Curvature Bounds for the Cost of Coherence\\in Conditional Optimal Transport}
\author{Xiaozhen~WANG\thanks{Centre de Math\'ematiques Appliqu\'ees (CMAP),
  \'Ecole Polytechnique, Institut Polytechnique de Paris, Palaiseau, France.
  E-mail: \href{mailto:xiaozhen.wang@polytechnique.edu}{xiaozhen.wang@polytechnique.edu}.}
  \and Zihan~Zhang\thanks{Sorbonne Universit\'e, Paris, France.
  E-mail: \href{mailto:zihan.zhang@etu.sorbonne-universite.fr}{zihan.zhang@etu.sorbonne-universite.fr}.}
  \and Xiang~Zhou\thanks{Sorbonne Universit\'e, Paris, France.
  E-mail: \href{mailto:xiang.zhou@etu.sorbonne-universite.fr}{xiang.zhou@etu.sorbonne-universite.fr}.}}
\date{}
\begin{document}
\maketitle

\begin{abstract}
We study the gap between the integrated direction-wise kinetic minima
and the minimum expected Dirichlet energy of a joint particle realization
of a prescribed two-parameter family of conditional probability measures.
Under regular coarea assumptions with compact connected fibres and a
star-shaped parameter domain, we derive a curvature lower bound valid
for every admissible law of Sobolev maps realizing the prescribed smooth
positive conditional measures.
Our method combines the weighted-Poisson characterization of local
optimal-transport velocities with an exact energy-residual identity,
a weak Stokes formula, and explicit transport constructions.
The minimum excess vanishes exactly when the Lie-bracket curvature
vanishes throughout the parameter domain, whereas at a nonzero-curvature
centre it is of order $r^4$ as $r\to0$ on squares of half-side length $r$
with unnormalized area measure.
In a two-phase torus model fitted to electricity-demand data, numerical
tests illustrate this quartic order for explicit couplings, with radial
transport approximately halving the ordered construction's excess on
the tested small squares.

\end{abstract}

\section{Introduction}
\label{sec:introduction}
Realizing a prescribed two-parameter family of probability measures by
the same particles costs Dirichlet energy. We ask how much more energy
this joint realization requires than the sum of the two direction-wise
kinetic minima.
Quadratic optimal transport (OT) couples two prescribed distributions while
minimizing mean squared displacement \citep{Villani2009,Santambrogio2015}.
Its dynamic formulation minimizes kinetic action over the intervening
distributions and velocities \citep{BenamouBrenier2000}.
For a prescribed distribution curve, the least $L^2$ norm of a velocity
satisfying the continuity equation equals the Wasserstein metric speed
\citep{AmbrosioGigliSavare2008}.
Under suitable finite-energy assumptions, the superposition principle
represents this motion by a probability law on particle paths
\citep{Lisini2007,StepanovTrevisan2017}.
Particle laws are also the optimization variables in Brenier's
generalized least-action flows \citep{Brenier1989,AmbrosioFigalli2009}.
For several parameters, extended Monge--Kantorovich theory
\citep{Brenier2003} and
Wasserstein-valued harmonic maps \citep{Lavenant2019} provide energy
formulations in terms of maps and local velocity fields.

Lavenant's lifting framework compares the minimum energy of local
velocity fields with the minimum average energy of random maps
realizing the same distributions \citep{LavenantLifting2024}.
These energies need not agree. Lavenant identifies the formal local
compatibility condition \citep[Section~5.3]{Lavenant2019}; his lifting
framework also gives attainment under coercivity and finite energy,
and a radial construction of local approximate realizations
\citep[Proposition~3.13 and proof of Theorem~5.10]{LavenantLifting2024}.
We quantify this difference for smooth conditional measures on moving
level sets, with a lower bound for all admissible Sobolev realization laws and a
sharp small-square exponent.

Let $\mu_y$, $y=(s,t)\in U\subset\R^2$, be obtained by conditioning on
two statistics in a Riemannian state space $(\X,g)$.
Choose in each direction the least-energy velocity reproducing the
distribution change and moving particles between the corresponding
level sets. The resulting fields $v_s,v_t$ need not
describe one common particle motion: changing $s$ then $t$ may send
a particle to a different endpoint from the reverse order, although
both orders give the correct final distribution.
A coherent realization gives each particle a map $\sigma:U\to\X$,
assigning its position at every condition $y$.
A probability law $Q$ on these maps couples all prescribed distributions:
it is an infinite-marginal transport plan \citep{Pass2013,Pass2015}.
We allow Sobolev maps with finite expected energy. The two local
minimizers give the benchmark
\[
 \Eloc=\frac12\int_U\int_\X
       \bigl(|v_s|_g^2+|v_t|_g^2\bigr)\dif\mu_y\dif y.
\]
Our problem is to find the least additional action needed for a joint
realization:
\[
 \begin{aligned}
 \Dcoh=\inf_Q\quad&
 \frac12\E_Q\int_U
 \bigl(|\partial_s\sigma|_g^2+|\partial_t\sigma|_g^2\bigr)\dif y-\Eloc,\\
 \text{subject to}\quad&
 \operatorname{Law}_Q(\sigma(y))=\mu_y
                  \quad\text{for a.e. }y\in U.
 \end{aligned}
\]
Here $\E_Q$ averages over particle maps, and
$\operatorname{Law}_Q(\sigma(y))$ is the distribution of their values at $y$.
The squared weak derivatives
are integrated over parameter area, with unnormalized measure
$\dif y=\dif s\dif t$, not along a single path.
\Cref{def:coupling} gives the precise
Sobolev formulation of the marginal constraint.
In a product chamber with product metric, both energies contain the
same base contribution $|U|$. Subtracting it gives the local-field and
random-map Dirichlet energies on the fixed fibre. The gap is therefore
the fibre-valued version of the Lagrangian--Eulerian lifting gap, with
the conditional geometry determining the local minimizers.

Under the regular coarea assumptions of \Cref{subsec:coarea-setting},
the Lie bracket $\Omega=[v_s,v_t]$ measures the noncommutativity of the
local velocities and gives a lower bound for every admissible
Sobolev coupling's excess action (\Cref{thm:certificate}).
At a centre with nonzero curvature, this bound and successive coordinate
transports give matching $r^4$ estimates for the minimum excess on
squares of half-side length $r$ (\Cref{thm:quartic-action}).
The constants depend on the local smooth data. Thus optimizing over all
joint realization laws, including rough or randomized realizations,
cannot remove the energy cost of noncommuting local OT velocities.
The proof expresses excess action as a squared $L^2$ residual and controls
curvature by that residual through a weak Stokes formula on Sobolev
sections. We verify attainment using the lifting compactness argument.
The minimum excess is zero
exactly when $\Omega$ vanishes, and the minimizing law is then unique
(\Cref{thm:flatness}). The smooth realization in this case is the
Frobenius integrability mechanism for commuting fields. The leading
displacement between the two coordinate orders
(\Cref{thm:order-defect}) and an explicit nonflat torus example
(\Cref{prop:torus-example}) make the obstruction concrete.

Conditional generators give a practical instance of the joint-law
problem: a fixed random seed follows one sample across conditions.
Shared randomness also underlies climate-XVA comparisons, where
climate and baseline simulations use common shocks
\citep{WangBuetGolfouse2026ClimateDyna}.
For smooth positive laws in our setting, weighted Poisson equations
determine the local OT velocities. The generator's derivatives
may differ from these velocities even when both reproduce exactly the
same laws. Thus curvature concerns the local OT fields, while the
generator's excess action measures the cost of its particular coupling.
Flow-based density estimation constructs maps to a reference law through
likelihood-increasing particle dynamics \citep{TabakVandenEijnden2010}.
Tabak, Trigila, and Zhao develop normalizing-flow methods for conditional
barycenters, mapping conditional laws to a common reference family while
retaining dependence on selected covariates \citep{TabakTrigilaZhao2024}.
Related methods include Sobolev descent \citep{MrouehSercuRaj2019},
conditional maps learned from paired populations \citep{BunneKrauseCuturi2022},
entropic conditional transport \citep{BaptistaEtAl2025}, and dynamic
conditional OT \citep{KerriganEtAl2024}.
Generative transfer constructs new entropic couplings from a reference
coupling under a shared unknown cost \citep{DebouchageWangRenBuetGolfouse2026};
path-independent flow matching learns compatible multi-parameter dynamics
\citep{TellezEtAl2026}.

Otto's Riemannian interpretation \citep{Otto2001}, the parallel transport
of Ambrosio and Gigli \citep{AmbrosioGigli2008}, and the second-order
calculus of Lott and Gigli \citep{Lott2008,Gigli2012} describe Wasserstein
geometry, with a finite-dimensional Gaussian setting \citep{Takatsu2011}.
Khesin and Lee \citep{KhesinLee2009} study a nonholonomic transport geometry.
Optimal-transport holonomy describes particle rearrangement after a
closed distribution path, for Gaussian families \citep{AbdelgalilGeorgiou2026}
and smooth transports \citep{AbdelgalilGeorgiouSmooth2026}.
Our Lie-bracket curvature concerns particle motions, rather than the
Riemannian curvature tensor of Wasserstein space.
Computational counterparts include dynamic transport and Dirichlet problems
on triangulated surfaces \citep{LavenantEtAl2018}.
Disintegrated transport compares two measure families within metric
fibres \citep{KitagawaTakatsu2024}; we seek one family's joint realization
across fibres.
Moser's theorem \citep{Moser1965} and regular representations of measure
families \citep{JostEtAl2019} give smooth realizations; matching both
kinetic minimizers is an additional constraint.
Schr\"odinger bridges \citep{Leonard2014} and entropic transport
\citep{PeyreCuturi2019} add an entropy cost.
Related variants impose martingale constraints
\citep{ChenConfortiRenWang2026} or use nonlinear convex functionals
of the coupling \citep{KazeykinaRenWangZhang2025}. We retain unregularized quadratic action.

Our numerical example fits a smooth density for two Fourier phases of
electricity-demand profiles and uses a fixed decoder to display samples.
For this fitted torus family, we compute the local OT fields, test the
curvature prediction for displacement, and compare explicit couplings
without changing the marginals. Radial transport approximately halves
the ordered construction's excess on the tested small squares; a
numerically evaluated curvature bound gives a lower benchmark. These
calculations concern the fitted phase model and its coordinate metric,
not the full distribution of demand profiles or the unrestricted minimum.
\Cref{sec:main-results} states the results,
\Cref{sec:numerical-experiments} presents the explicit torus benchmark
and a fitted conditional phase model, and \Cref{sec:proofs} gives the proofs.
We use $\partial_i$ for parameter derivatives, $\mathrm d$ for
differentials and exterior derivatives, and $\nabla_VW$ for covariant
differentiation with the Levi-Civita connection of $g$.
For a scalar function $f$, $\nabla f$ denotes its Riemannian gradient;
subscripts or superscripts specify a fibre or a different metric.

\section{Main Results}
\label{sec:main-results}
We now formulate the action-gap problem for conditional measures on
smooth level sets. We first specify the geometry, the local OT velocities,
and the admissible joint couplings. We then state the curvature bounds
and the constructions that attain their small-square order.

\subsection{Conditional measures and transport energies}
\label{subsec:coarea-setting}

A particle state $x$ determines two quantities $H(x)=(s,t)$.
Conditioning on them gives the family to be transported. We fix the
Euclidean parameter coordinates and measure particle velocities with $g$.

\begin{assumption}[Regular coarea chamber]
\label{ass:geometry}
The manifold $(\X,g)$ is smooth, connected, complete, and without boundary.
There are an open star-shaped set $U^+\subset\R^2$ and a bounded open
star-shaped Lipschitz domain $U\subset\R^2$ with
$\overline U\Subset U^+$. There is a smooth map $H:\X\to\R^2$ whose
restriction $H:H^{-1}(U^+)\to U^+$ is a surjective proper submersion with
connected fibres. The reference probability has a smooth strictly positive
density $p$ with respect to Riemannian volume on $\X$.
\end{assumption}

For a prescribed value $y$, the possible particle states form the fibre
$M_y=H^{-1}(y)$. These fibres are compact and have no boundary.
We write $E=H^{-1}(U)$ for their union over the parameter domain and
$E^+=H^{-1}(U^+)$ for the larger regular region.
The coarea formula \citep{Federer1969} gives the conditional probability
\begin{equation}
 \dif\mu_y(x)=\frac{p(x)}{Z(y)J_H(x)}\dif\mathrm{vol}_{M_y}(x),
 \label{eq:coarea}
\end{equation}
where $\mathrm{vol}_{M_y}$ is the volume induced by $g$ and
\[
 J_H(x)=\sqrt{\det(\mathrm dH_x(\mathrm dH_x)^*)},\qquad
 Z(y)=\int_{M_y}\frac{p}{J_H}\dif\mathrm{vol}_{M_y}.
\]
Here $\mathrm dH_x$ is the differential, its adjoint uses $g$ and the
Euclidean base metric, and $Z(y)$ normalizes the measure. Define
\begin{equation}
 \int_E f\dif\mubar
 =\int_U\int_{M_y}f(x)\dif\mu_y(x)\dif y.
 \label{eq:occupation-measure}
\end{equation}
Thus $\mubar(E)=|U|$, where $|U|$ is the Euclidean area of $U$ and
$\dif y=\dif s\dif t$. Norms use $g$ and its dual; $L^2$ norms use
$\mubar$, or $\mu_y$ when restricted to a fibre. The tangent bundle
$TE$ consists of the spaces $T_xE$, $x\in E$.
The space $L^2(E,\mubar;TE)$ consists of measurable vector fields
$V(x)\in T_xE$ with $\int_E|V|_g^2\dif\mubar<\infty$.

Write $e_s=(1,0)$ and $e_t=(0,1)$ for the parameter coordinate vectors.
For $i\in\{s,t\}$, choose the least-energy velocity producing the
prescribed change in direction $e_i$:
\begin{equation}
 \begin{aligned}
 \underset{V\in L^2(E,\mubar;TE)}{\operatorname{minimize}}\quad&
       \frac12\int_E |V|_g^2\dif\mubar\\
 \text{subject to}\quad&\mathrm dH(V)=e_i\quad\mubar\text{-a.e.},\\
 &\int_E\dif\varphi(V)\dif\mubar=0
              \quad\text{for all }\varphi\in C_c^1(E).
 \end{aligned}
 \label{eq:local-minimum}
\end{equation}
Here $C_c^1(E)$ denotes the continuously differentiable functions with
compact support in $E$. The constraints prescribe, respectively, the particle's condition and
the transported distribution. They are imposed separately in each direction.
The normal vector $n_i=(\mathrm dH)^*(\mathrm dH(\mathrm dH)^*)^{-1}e_i$
minimizes the first constraint alone. Every other feasible velocity
differs by a vector in the vertical tangent bundle $\V=\ker\mathrm dH$,
whose fibre at $x$ is $\V_x=T_xM_{H(x)}\subset T_xE$.
Writing $\nabla_{M_y}$ for the fibre gradient, with $y=H(x)$, solve
\begin{equation}
 \begin{gathered}
 \int_E g(\nabla_{M_y} u_i,\nabla_{M_y}\varphi)\dif\mubar
 =-\int_E\dif\varphi(n_i)\dif\mubar\quad(\varphi\in C_c^1(E)),\\
 \int_{M_y}u_i\dif\mu_y=0\quad\text{a.e. }y.
 \end{gathered}
 \label{eq:poisson}
\end{equation}
Tests are defined on the total space. The notation $L^2(U;H^1(M_y))$
means that $u$ and its fibre gradient are square integrable under
$\mubar$; \Cref{ass:local} supplies fixed-fibre coordinates.

\begin{proposition}[Local kinetic projection]
\label{prop:local-ot}
Under \Cref{ass:geometry},
\eqref{eq:poisson} has a unique solution in $L^2(U;H^1(M_y))$ with the
stated normalization. It has a smooth representative extending to $E^+$,
obtained by solving the same normalized fibre problem there. The fields
\begin{equation*}
 v_i=n_i+\nabla_{M_y} u_i
\end{equation*}
are smooth and have finite occupation energy. Each $v_i$ is the unique
minimizer of \eqref{eq:local-minimum}, and every feasible $V$ satisfies
\begin{equation*}
 \|V\|_{L^2(\mubar)}^2
 =\|v_i\|_{L^2(\mubar)}^2+\|V-v_i\|_{L^2(\mubar)}^2.
\end{equation*}
\end{proposition}

The direction-wise kinetic benchmark is
\begin{equation}
 \Eloc=\frac12\sum_{i=s,t}\int_E|v_i|_g^2\dif\mubar.
 \label{eq:local-energy}
\end{equation}

Let $d_g$ be the ambient Riemannian distance on $(\X,g)$.
For probabilities $\nu_0,\nu_1$ with finite second moments, the quadratic
Wasserstein distance is
\begin{equation*}
 W_2(\nu_0,\nu_1)^2
 =\inf_{\pi\in\Pi(\nu_0,\nu_1)}
       \int_{\X\times\X}d_g(x,x')^2\dif\pi(x,x'),
\end{equation*}
where $\Pi(\nu_0,\nu_1)$ denotes the joint probability measures with
marginals $\nu_0,\nu_1$. In contrast with a two-endpoint problem, the
local kinetic problem keeps the entire family $\mu_y$ fixed. Its relation
to quadratic optimal transport is the following metric-speed identity.

\begin{proposition}[Identification with Wasserstein speed]
\label{prop:wasserstein-speed}
Under \Cref{ass:geometry}, for every $y\in U$ and $i\in\{s,t\}$,
\begin{equation*}
 \lim_{h\to0}\frac{W_2(\mu_{y+he_i},\mu_y)}{|h|}
 =\left(\int_{M_y}|v_i|_g^2\dif\mu_y\right)^{1/2},
\end{equation*}
where $h$ is small enough that the coordinate segment lies in $U$.
Each coordinate restriction is locally Lipschitz in $W_2$, and
\begin{equation}
\label{eq:wasserstein-dirichlet}
 \Eloc=\frac12\int_U\sum_{i=s,t}
 \left(\lim_{h\to0}\frac{W_2(\mu_{y+he_i},\mu_y)}{|h|}\right)^2\dif y.
\end{equation}
Thus $\Eloc$ is the Wasserstein Dirichlet energy of the prescribed family.
\end{proposition}

Thus the local problem gives the infinitesimal quadratic transport cost
along each coordinate line. A finite coordinate flow follows the
prescribed intermediate measures, which need not form a Wasserstein
geodesic.

To couple all $\mu_y$ at once, let one particle follow a section
$\sigma:U\to\X$ with $H(\sigma(y))=y$. A law $Q$ on sections is a
transport plan with infinitely many marginals. Sobolev sections allow
limits with bounded action.
Properness makes $K=H^{-1}(\overline U)$ compact.
Fix a smooth Euclidean embedding of a neighbourhood of $K$ into $\R^N$
\citep[Chapter~6]{Lee2012} and identify $K$ with its image.
Here $H^1(U;\R^N)$ consists of square-integrable maps whose first weak
derivatives are square integrable, and $\operatorname{id}_U$ is the
identity map on $U$. Define
\begin{equation}
 \mathcal S=\{\sigma\in H^1(U;\R^N):
 \sigma(y)\in K\ \text{a.e.},\quad
 H\circ\sigma=\operatorname{id}_U\ \text{a.e.}\}.
 \label{eq:section-space}
\end{equation}
Equip $\mathcal S$ with the strong $H^1$ Borel structure. Its elements
are almost-everywhere classes, denoted by $[\sigma]$ when this distinction
is needed, so we impose marginals in integrated form.
\Cref{lem:representatives} supplies jointly Borel representatives.
Weak derivatives are denoted by $\partial_i\sigma$; their norms use $g$.

\begin{definition}[Admissible coupling and coherent action]
\label{def:coupling}
The set $\Qcoh$ consists of the Borel probabilities $Q$ on $\mathcal S$
with finite action
\begin{equation}
 \Aact(Q)=\frac12\int_{\mathcal S\times U}
 \bigl(|\partial_s\sigma(y)|_g^2+|\partial_t\sigma(y)|_g^2\bigr)
 \dif Q(\sigma)\dif y<\infty
 \label{eq:action}
\end{equation}
and the occupation constraint
\begin{equation}
 \int_{\mathcal S\times U}f(y,\sigma(y))\dif Q(\sigma)\dif y
 =\int_U\int_{M_y}f(y,x)\dif\mu_y(x)\dif y
 \label{eq:occupation}
\end{equation}
for every bounded Borel $f:U\times E\to\R$. Set
\begin{equation}
 \Dcoh=\inf_{Q\in\Qcoh}\Aact(Q)-\Eloc.
 \label{eq:gap}
\end{equation}
\end{definition}

Disintegration gives marginal $\mu_y$ for almost every $y$.
A generator with measurable $H^1$ sections, these marginals, and finite
action induces such a law; changing its seed parametrization need not
change $Q$.
For $Q\in\Qcoh$, its velocity residuals are the vertical fields on the
product space
\begin{equation}
 r_i(\sigma,y)=\partial_i\sigma(y)-v_i(\sigma(y)),\qquad i=s,t.
 \label{eq:residual}
\end{equation}

\begin{proposition}[Feasibility and exact residual identity]
\label{prop:residual}
Under \Cref{ass:geometry}, the set
$\Qcoh$ is nonempty, $0\leq\Dcoh<\infty$, and every $Q\in\Qcoh$ satisfies
\begin{equation}
 \Aact(Q)-\Eloc
 =\frac12\int_{\mathcal S\times U}(|r_s|_g^2+|r_t|_g^2)
 \dif Q\dif y.
 \label{eq:residual-identity}
\end{equation}
\end{proposition}

The identity follows from the local projection in \Cref{prop:local-ot}: the average velocity
induced by $Q$ is feasible, so its correction is orthogonal to $v_i$.
The excess action is therefore the cost of realizing both directions
by one law.

\begin{proposition}[Existence of an optimal coherent coupling]
\label{prop:coherent-attainment}
Under \Cref{ass:geometry}, there exists $Q_*\in\Qcoh$ such that
\begin{equation*}
 \Aact(Q_*)=\inf_{Q\in\Qcoh}\Aact(Q)=\Eloc+\Dcoh.
\end{equation*}
If $(Q_n)$ is a minimizing sequence, the laws obtained by regarding its
sections as $L^2(U;\R^N)$ classes have a narrowly convergent subsequence
on the strong $L^2$ space;
the limiting law corresponds to an optimal Borel law on $\mathcal S$.
\end{proposition}

Here narrow convergence means convergence of integrals against all
bounded continuous functions on the stated $L^2$ space.
This is the compact-fibre version of the lifting attainment argument
in Lavenant~\citep[Proposition~3.13]{LavenantLifting2024}. We give a self-contained
proof using $L^2$ compactness of bounded-action sections and closure of
the integrated marginal constraint. The curvature estimate in \Cref{thm:certificate}
bounds the value of this minimization problem.

\subsection{Curvature and the coherence gap}
\label{subsec:curvature-certificate}

The two local velocities need not commute. Let $[V,W]$ denote the Lie
bracket of vector fields. The Levi-Civita connection $\nabla$ of $g$
is torsion-free, so $[V,W]=\nabla_VW-\nabla_WV$.
We define the transport curvature by
\begin{equation}
 \Omega=[v_s,v_t].
 \label{eq:curvature}
\end{equation}
It measures the failure of the two transport directions to commute and
is vertical since $\mathrm dH(v_i)=e_i$ is constant.
To estimate the gap, use a compactly supported $C^1$ section
$\alpha\in C_c^1(E;\V^*)$ of the dual vertical bundle:
$\alpha_x$ is a linear functional on $\V_x$.
Extend it to a one-form on $E$ that annihilates $v_s,v_t$:
\begin{equation*}
 \widetilde\alpha_x(W)=\alpha_x\bigl(W-(\mathrm dH_x(W))_s v_s(x)
                                      -(\mathrm dH_x(W))_t v_t(x)\bigr).
\end{equation*}
Here $\dif\widetilde\alpha$ is the exterior derivative, a two-form.
The bound uses the curvature integral and its mixed and vertical components:
\begin{align}
 I_\alpha&=\left|\int_E\alpha(\Omega)\dif\mubar\right|,
 &B_\alpha^2&=\int_E\sum_{i=s,t}
       |\dif\widetilde\alpha(v_i,\cdot)|_{\V^*}^2\dif\mubar,
 \nonumber\\
 A_\alpha&=\frac12\|\dif\widetilde\alpha|_{\V\times\V}
                          \|_{L^\infty(\mathrm{comass})}.
 \label{eq:witness-terms}
\end{align}
Here $B_\alpha\geq0$; comass means the supremum on pairs of vertical
vectors of norm at most one, and essential suprema use $\mubar$.

\begin{theorem}[A curvature lower bound for the action gap]
\label{thm:certificate}
Under \Cref{ass:geometry}, every $Q\in\Qcoh$ and
$\alpha\in C_c^1(E;\V^*)$ satisfy
\begin{equation}
 I_\alpha\leq B_\alpha\sqrt{2(\Aact(Q)-\Eloc)}
                    +2A_\alpha(\Aact(Q)-\Eloc).
 \label{eq:certificate}
\end{equation}
Consequently,
\begin{equation}
 \Dcoh\geq
 \sup_{\substack{\alpha\in C_c^1(E;\V^*)\\ I_\alpha>0}}
 \frac{2I_\alpha^2}
 {(B_\alpha+\sqrt{B_\alpha^2+4A_\alpha I_\alpha})^2}.
 \label{eq:certificate-lower}
\end{equation}
The supremum is taken to be zero when there is no such test form.
For $I_\alpha>0$ the denominator is positive, so one such form implies
$\Dcoh>0$.
\end{theorem}

The mechanism of the bound is an integration-by-parts identity on each
Sobolev particle surface. Stokes' formula pairs curvature with terms
linear and quadratic in the velocity residuals; compact support removes
boundary terms. Averaging over the law and applying Cauchy--Schwarz
then bounds these terms by the excess action in
\eqref{eq:residual-identity}. Applying the identity before averaging is
what permits arbitrary Sobolev coupling laws
(\Cref{subsec:proof-stokes}).

For a parametrized map $F$, write $F_y=F(y,\cdot)$.
For a measurable map $\Psi$ and a measure $\nu$, the pushforward
$\Psi_\#\nu$ is defined by $(\Psi_\#\nu)(B)=\nu(\Psi^{-1}(B))$
for every measurable set $B$ in the target space.

\begin{proposition}[Flatness and the zero-gap coupling]
\label{thm:flatness}
Under \Cref{ass:geometry},
\begin{equation*}
 \Dcoh=0\quad\Longleftrightarrow\quad\Omega\equiv0\text{ on }E.
\end{equation*}
When these conditions hold, every star centre $y_\star\in U$ determines a
unique smooth map $F:U\times M_{y_\star}\to E$ satisfying
\begin{equation*}
 F(y_\star,z)=z,\qquad H(F(y,z))=y,\qquad
 \partial_iF(y,z)=v_i(F(y,z))\quad(i=s,t).
\end{equation*}
It has a smooth $E^+$-valued extension near
$\overline U\times M_{y_\star}$, with bounded first derivatives on this
compact set.
Each $F_y$ is a diffeomorphism onto $M_y$ and
$(F_y)_\#\mu_{y_\star}=\mu_y$. The probability
\begin{equation*}
 Q_F=\bigl(z\mapsto [F(\cdot,z)]\bigr)_\#\mu_{y_\star}
\end{equation*}
is the unique law in $\Qcoh$ attaining the infimum, and
$\Aact(Q_F)=\Eloc$. For every piecewise $C^1$ path $\gamma$ in $U$,
horizontal transport has endpoint map
$F_{\gamma(1)}\circ F_{\gamma(0)}^{-1}$.
\end{proposition}

Flatness adds smoothness and uniqueness to the general existence result in \Cref{prop:coherent-attainment}.
Uniqueness concerns the law, not a seed parametrization.

For $k$ controls, replace $\R^2$ by $\R^k$ in
\Cref{ass:geometry}. The local minimizers $v_1,\ldots,v_k$
give pairwise curvatures $\Omega_{ij}=[v_i,v_j]$, $i<j$.
A smooth family of fibre diffeomorphisms realizing all $v_i$ forces
every bracket to vanish by equality of mixed derivatives. Conversely,
vanishing of all brackets makes the radial-flow proof of
\Cref{thm:flatness} apply in all $k$ directions, with the
prescribed marginals. Our quantitative action bounds remain
two-dimensional.

For an upper bound, transport reference particles first in $s$, then
in $t$. This gives a coherent law whose excess has the same quartic
order as the lower bound at a nonflat centre.
Fix $y_0=(s_0,t_0)\in U$ and write
\[
 U_r=y_0+(-r,r)^2,
 \qquad
 \kappa(y)=\left(\int_{M_y}|\Omega(x)|_g^2\dif\mu_y(x)\right)^{1/2}.
\]
An argument $U_r$ means restriction to this square, with unnormalized
area measure. Uniqueness of the fibre problems leaves $v_s,v_t$ unchanged.
Let $\Fl_{v_i}^a$ denote the local flow of $v_i$ at signed time $a$.
For sufficiently small $r>0$, $y=(s,t)\in U_r$, and $z\in M_{y_0}$, define
\begin{equation}
\label{eq:axial-realization}
 F_r(y,z)=\Fl_{v_t}^{\,t-t_0}\circ
          \Fl_{v_s}^{\,s-s_0}(z),
 \qquad
 Q_r=\bigl(z\mapsto F_r(\cdot,z)\bigr)_\#\mu_{y_0}.
\end{equation}
This construction gives one section for each reference particle. Its
$t$ derivative is $v_t$, while the correction to its $s$ derivative
will be estimated in terms of curvature.

\begin{theorem}[Two-sided quartic action bounds]
\label{thm:quartic-action}
Under \Cref{ass:geometry}, fix $y_0\in U$. There exist $r_*>0$ and
$L<\infty$, depending only on the smooth data on a compact neighbourhood of
$M_{y_0}$, such that $\overline{U_{r_*}}\subset U$ and the following hold
for every $0<r\leq r_*$.
The law $Q_r$ in \cref{eq:axial-realization} is admissible on $U_r$, and
\begin{equation}
\label{eq:quartic-upper}
 0\leq\Dcoh(U_r)
 \leq\Aact_{U_r}(Q_r)-\Eloc(U_r)
 \leq\frac23 e^{2Lr}r^4\sup_{y\in U_r}\kappa(y)^2.
\end{equation}
If $\kappa(y_0)>0$, there are $r_0\in(0,r_*]$ and $c,C>0$,
independent of $r$, such that for every $0<r\leq r_0$,
\begin{equation}
\label{eq:quartic-sandwich}
 c\,\kappa(y_0)^2 r^4
 \leq\Dcoh(U_r)
 \leq\Aact_{U_r}(Q_r)-\Eloc(U_r)
 \leq C\,\kappa(y_0)^2 r^4.
\end{equation}
The constants and $r_0$ may depend on $y_0$, the local geometry, the density,
and their derivatives. No uniform neighbourhood size is asserted as
$\kappa(y_0)$ tends to zero.
\end{theorem}

These bounds establish the sharp quartic order of the minimum and show
that the ordered construction's excess is at most $C/c$ times the optimum
on small squares.

\subsection{Ordered transports}
\label{subsec:order-defect}

We can also compare one particle's endpoints under the two coordinate
orders. Fix $y_0\in U$. For small signed $\varepsilon$, both paths to
$y_0+\varepsilon(e_s+e_t)$ lie in $U$. On $M_{y_0}$ set
\begin{equation}
 P_{st}^{\varepsilon}
 =\operatorname{Fl}_{v_t}^{\varepsilon}\circ
   \operatorname{Fl}_{v_s}^{\varepsilon},\qquad
 P_{ts}^{\varepsilon}
 =\operatorname{Fl}_{v_s}^{\varepsilon}\circ
   \operatorname{Fl}_{v_t}^{\varepsilon},
 \label{eq:ordered-maps}
\end{equation}
and, using the ambient distance $d_g$, set
\begin{equation}
 \delta_{\mathrm{loop}}(\varepsilon)
 =\left(\int_{M_{y_0}}
 d_g(P_{st}^{\varepsilon}(x),P_{ts}^{\varepsilon}(x))^2
 \dif\mu_{y_0}(x)\right)^{1/2}.
 \label{eq:order-defect}
\end{equation}

\begin{proposition}[Expansion of the order defect]
\label{thm:order-defect}
Under \Cref{ass:geometry}, there are
$C<\infty$ and $\varepsilon_0>0$, depending on $y_0$ and the smooth data
on a compact neighbourhood of its fibre, such that for
$|\varepsilon|<\varepsilon_0$ both maps in \eqref{eq:ordered-maps} are
diffeomorphisms from $M_{y_0}$ onto
$M_{y_0+\varepsilon(e_s+e_t)}$ and
\begin{equation}
 (P_{st}^{\varepsilon})_\#\mu_{y_0}
 =(P_{ts}^{\varepsilon})_\#\mu_{y_0}
 =\mu_{y_0+\varepsilon(e_s+e_t)}.
 \label{eq:ordered-marginals}
\end{equation}
Moreover,
\begin{equation}
 \left|\delta_{\mathrm{loop}}(\varepsilon)
 -\varepsilon^2\|\Omega|_{M_{y_0}}\|_{L^2(\mu_{y_0};g)}\right|
 \leq C|\varepsilon|^3.
 \label{eq:order-expansion}
\end{equation}
The same assertion holds if the distance in \eqref{eq:order-defect} is
replaced by the intrinsic distance of the common endpoint fibre, with a
possibly different constant $C$.
\end{proposition}

\begin{corollary}[Squared defect]
\label{cor:squared-defect}
With the notation and hypotheses of \Cref{thm:order-defect},
\begin{equation*}
 \delta_{\mathrm{loop}}(\varepsilon)^2
 =\varepsilon^4\int_{M_{y_0}}|\Omega|_g^2\dif\mu_{y_0}
   +O(|\varepsilon|^5).
\end{equation*}
The constant is independent of sufficiently small signed $\varepsilon$.
\end{corollary}

\section{Examples and Numerical Experiments}
\label{sec:numerical-experiments}
\newcommand{\TestNLL}{2.535}
\newcommand{\BaselineNLL}{2.866}
\newcommand{\KappaMin}{0.162}
\newcommand{\KappaMax}{1.337}
\newcommand{\LoopError}{1.01}
\newcommand{\LoopErrorMax}{2.67}
\newcommand{\AxialSlope}{4.004}
\newcommand{\AxialSlopeMin}{3.987}
\newcommand{\AxialSlopeMax}{4.090}
\newcommand{\GeneratorSlope}{2.000}
\newcommand{\SamplerMoment}{1.9\times10^{-3}}
\newcommand{\MomentError}{3.6\times10^{-7}}
\newcommand{\IdentityError}{3.6\times10^{-14}}

\newcommand{\RadialReduction}{50.0}
\newcommand{\RadialReductionMin}{49.97}
\newcommand{\RadialReductionMax}{50.17}
\newcommand{\RadialSlope}{4.005}
\newcommand{\WitnessRatioMin}{1.203}
\newcommand{\WitnessRatioMax}{1.209}
\newcommand{\WitnessRatioMedian}{1.205}

\newcommand{\RepresentationRMSE}{0.62329}
\newcommand{\SupplementSamplerS}{9.482\times10^{-4}}
\newcommand{\SupplementSamplerT}{6.658\times10^{-5}}
\newcommand{\SupplementSamplerMoment}{1.875\times10^{-3}}
\newcommand{\SupplementSpatial}{9.386\times10^{-10}}
\newcommand{\SupplementParameter}{5.791\times10^{-8}}
\newcommand{\SupplementIdentity}{3.644\times10^{-14}}
\newcommand{\SupplementAnalytic}{1.491\times10^{-12}}
\newcommand{\SupplementRadialRefinement}{2.996\times10^{-6}}
\newcommand{\SupplementWitnessRefinement}{8.852\times10^{-8}}
\newcommand{\SupplementFurtherRadial}{9.205\times10^{-6}}
\newcommand{\SupplementFurtherWitness}{1.091\times10^{-7}}
\newcommand{\SupplementExtensionMoment}{4.071\times10^{-8}}
\newcommand{\SupplementExtensionIdentity}{0.9827}

An explicit torus family provides an analytic benchmark for the solver.
We then compare joint realizations of a conditional phase model fitted
to electricity-demand data.

\subsection{An explicit nonflat torus benchmark}
\label{subsec:torus-benchmark}
Two interacting Fourier modes give nonzero curvature and a positive
coherence gap for smooth positive densities on a product space.
Let $\T^2=(\R/2\pi\mathbb Z)^2$ with its standard flat metric and
coordinates $x=(x_1,x_2)$, and write
$\omega=(2\pi)^{-2}\dif x_1\dif x_2$ for normalized Haar measure. Fix
$a,b\neq0$ and choose $\delta>0$ with $\delta(|a|+|b|)<1$. On
$U_\delta=(-\delta,\delta)^2$ consider the fibre probability densities
\begin{equation}
 \rho_{s,t}(x)=1+as\cos x_1+bt\cos(x_1+x_2),\qquad
 \dif\mu_{s,t}(x)=\rho_{s,t}(x)\dif\omega(x).
 \label{eq:torus-density}
\end{equation}
We identify these with probabilities on
$\{(s,t)\}\times\T^2$ in the product space $\R^2\times\T^2$.
In the decomposition $v_i=e_i\oplus w_i$, $w_i$ is tangent to the torus;
vertical vectors such as $\Omega$ are displayed through their two fibre components.

\begin{proposition}[A coarea example with nonzero curvature]
\label{prop:torus-example}
The family \eqref{eq:torus-density} is realized by coarea conditioning a
smooth strictly positive reference probability on
$\X=\R^2\times\T^2$, with product metric and $H(y,x)=y$, and satisfies
\Cref{ass:geometry} on $U_\delta$.
Writing $v_i=e_i\oplus w_i$, at $y=0$ one has
\begin{equation*}
 w_s(0,x)=-a(\sin x_1,0),\qquad
 w_t(0,x)=-\frac b2\sin(x_1+x_2)(1,1),
\end{equation*}
and
\begin{equation}
 \Omega(0,x)=ab
 \begin{pmatrix}
 -\frac12\sin x_2-\frac1{10}\sin(2x_1+x_2)\\[1mm]
 \frac15\sin(2x_1+x_2)
 \end{pmatrix},\qquad
 \|\Omega(0,\cdot)\|_{L^2(\omega)}^2=\frac3{20}a^2b^2.
 \label{eq:torus-curvature}
\end{equation}
Consequently $0<\Dcoh(U_\delta)<\infty$, and at the base point $y_0=0$,
\begin{equation}
 \delta_{\mathrm{loop}}(\varepsilon)=|ab|\sqrt{\frac3{20}}\,\varepsilon^2
                       +O(|\varepsilon|^3).
 \label{eq:torus-defect}
\end{equation}
\end{proposition}

For $a=b=1$ at $y=0$, a $33^2$ Fourier-grid calculation gives a discrete
$L^2(\omega)$ curvature error of $\SupplementAnalytic$ against
\eqref{eq:torus-curvature}; the computed squared norm agrees with $3/20$
to 14 decimal places. A constant-density control gives zero curvature
and phase velocities to reported precision. Supplementary Information~1--2
give the implementation and accuracy checks.

\subsection{A fitted conditional phase model}
For a conditional density fitted to electricity-demand data, we compute
the local OT fields, their curvature, and the excess actions of several
joint realizations. The fitted marginals remain fixed: curvature tests
whether any realization can attain the local OT benchmark, while
\eqref{eq:residual-identity} measures the excess of each coupling.

We use the UCI household electricity-demand data \citep{HebrailBerard2006}.
The 1,358 complete days are averaged into 96 quarter-hour values and split
chronologically into 950 training, 203 validation, and 205 test days.
From each log-load profile we extract its mean, log standard deviation,
and the phases $\theta=(\theta_1,\theta_2)\in\T^2$ of its first two Fourier modes.
The standardized mean and log standard deviation form $y=(s,t)$,
respectively; standardization uses training
data only. We fit the conditional phase density
\begin{equation*}
 \rho_y(\theta)
 =\frac{\exp\{a(y)\cdot f(\theta)\}}
 {\int_{\T^2}\exp\{a(y)\cdot f(\xi)\}\dif\omega(\xi)}.
\end{equation*}
Here $f:\T^2\to\R^{12}$ collects the sine and cosine features at frequencies
$(1,0)$, $(0,1)$, $(1,1)$, $(1,-1)$, $(2,0)$, $(0,2)$;
$a:\R^2\to\R^{12}$ is a $2$--$16$--$12$ tanh
network. We minimize negative log-likelihood (NLL) with Adam (step $0.008$,
weight decay $0.002$, at most 1,200 full-batch epochs). Validation
likelihood selects checkpoints separately for seeds 11, 23, and 37.
Reported NLLs use the Lebesgue phase density $\rho_y/(2\pi)^2$.
Mean test NLL is $\TestNLL$ nats, compared with
$\BaselineNLL$ for a condition-independent density using the same features.

The conditional model is a density on two phases, with a fixed decoder
for visualization in profile space. Successive inversion of the marginal
and conditional phase cumulative distribution functions (CDFs) gives
a shared-seed generator $G_y:(0,1)^2\to\T^2$
\citep{Rosenblatt1952}, with seed $z$ uniformly distributed on $(0,1)^2$.
Its decoded log-load profile is the sum of
two cosine modes, rescaled to the prescribed mean and standard deviation;
their energy ratio $0.418:0.582$ is fixed from training data.
Exponentiation gives positive demand profiles (\hyperref[fig:generator]{Fig.~\ref*{fig:generator}a}).
A 4,096-seed check at the median-control centre of fit 11 gives
Fourier-moment error $\SamplerMoment$. These errors are Euclidean norms
of the differences between the 12 generated and target feature moments.
This representation makes the conditional law and its generator explicit;
it does not fit the full 96-dimensional distribution of daily profiles.

On $\R^2\times\T^2$ with $H(y,\theta)=y$ and the
product metric, the reference law $\gamma(y)\dif y\,\rho_y\dif\omega$,
with $\gamma$ a positive smooth probability density on $\R^2$, satisfies
\Cref{ass:geometry}. Costs use the product metric on standardized
condition and phase coordinates, not a metric induced by decoded demand
profiles. The identical base contributions cancel in excess actions.
Test-day reconstruction from observed controls and phases gives log-load
root mean squared error (RMSE) $\RepresentationRMSE$, measuring representation
loss, not sampling error.
Supplementary Information~1 gives per-fit scores and check coverage.

\begin{figure}[tb]
 \centering
 \includegraphics[width=\linewidth]{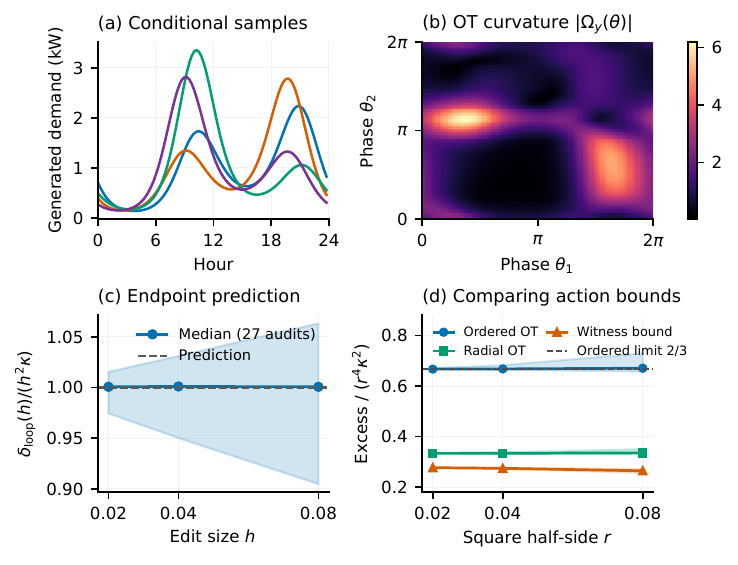}
 \caption{A fitted two-phase generator and its OT energy comparison.
 Here $\Omega_y(\theta)=\Omega(y,\theta)$ and $\kappa=\kappa(y_0)$ at
 each audit centre $y_0$.
 Panels (a,b) use seed 11 at the median-control centre; the four latent
 seeds in (a) are fixed in advance. Panels (c,d) report medians and
 10--90\% ranges over all 27 fit--centre pairs, not confidence intervals.
 Panel (d) compares two feasible excess actions with a numerical witness
 bound, normalized by $r^4\kappa^2$. The dashed $2/3$ line is the
 small-square coefficient of the ordered construction
 \eqref{eq:axial-leading-action}.
 The displayed curvature and costs concern the fitted phase model.}
 \label{fig:generator}
\end{figure}

We audit all nine combinations of the training-control 30th, 50th, and
70th percentiles, for each of the three fits. Centres are fixed before
curvature is computed. Fourier collocation solves the weighted Poisson
equations; parameter differences are used to evaluate the curvature in \eqref{eq:curvature},
and fourth-order Runge--Kutta (RK4) flows give the ordered updates.
Refinement from $33^2$ to $65^2$ changes curvature norms by less
than $10^{-6}$ relatively. The analytic and flat controls are reported in
\Cref{subsec:torus-benchmark}; Supplementary Information~1--2 specify
the solvers and accuracy checks.

For each audit centre $y_0$, write $\kappa=\kappa(y_0)$.
Across these 27 audits, the estimated $\kappa$ ranges from $\KappaMin$ to $\KappaMax$.
We run the two coordinate orders with the same weighted particles for
common coordinate increments $h=0.08,0.04,0.02$. At $h=0.02$, the normalized error
$|\delta_{\mathrm{loop}}(h)/(h^2\kappa)-1|$ has median
$\LoopError\%$ and maximum $\LoopErrorMax\%$. Thus the nonzero curvature
predicts particle displacement as in \eqref{eq:order-expansion}, although
both orders have the same
target marginal; the largest endpoint Fourier-moment error is
$\MomentError$.

On $U_r$, $r=0.08,0.04,0.02$, starting the ordered flows at $G_{y_0}(z)$
gives a second generator with the same fitted marginals in the continuum
construction. Each comparison is anchored at its own centre $y_0$.
Slopes are least-squares slopes of log excess action against $\log r$
over these three radii. The ordered excess-action slope has median $\AxialSlope$ and range
$[\AxialSlopeMin,\AxialSlopeMax]$, whereas the original generator's
median slope is $\GeneratorSlope$. These actions use tensor three-point
Gauss quadrature; implicit CDF differentiation gives the original
generator's velocity. Its excess computed from \eqref{eq:action} and
\eqref{eq:local-energy} agrees with the residual evaluation to $\IdentityError$.
The original generator's near-quadratic excess is consistent with a
nonzero central residual integrated over an area of order $r^2$.
For the ordered construction, the residual vanishes at the centre and
is $O(r)$ on the square, giving an integrated excess of $O(r^4)$.

We next compare the ordered construction with the radial maps of
\Cref{lem:radial-comparison}, which transport along the straight segment
from $y_0$ to each $y$. Both constructions have the same marginals.
The radial excess has median slope $\RadialSlope$; at $r=0.02$ it is
$\RadialReduction\%$ below the ordered cost, with reductions
$\RadialReductionMin$--$\RadialReductionMax\%$ over all 27 pairs.
Thus radial transport reduces the observed excess while retaining
quartic scaling on the tested squares.

We compare these costs with \eqref{eq:certificate-lower}, using
$\chi(y)=(1-(s-s_0)^2/r^2)(1-(t-t_0)^2/r^2)$ times the fibre one-form
obtained as the metric dual of the Fourier interpolant of the computed
central curvature, extended independently of $y$ in product coordinates.
The zero-trace factor is handled by bounded interior cutoffs that converge
in $H^1$ and preserve the bound on the vertical comass; Supplementary
Information~1 gives this extension of the witness class.
A Fourier coefficient sum bounds the vertical comass.
At $r=0.02$, radial cost divided by the evaluated lower bound ranges
from $\WitnessRatioMin$ to $\WitnessRatioMax$
(\hyperref[fig:generator]{Fig.~\ref*{fig:generator}d}). These are numerical evaluations of continuum
bounds, not a certified enclosure or a computation of the unrestricted
minimum. All 27 smallest squares are refined; failed accuracy checks
trigger further refinement, not exclusion. Refinement comparisons and
reported direct/residual checks agree within $1\%$ (Supplementary Information~1).
Moment and refinement checks concern the fitted laws; they do not
establish conditional calibration or curvature of the population law.

\section{Proofs}
\label{sec:proofs}
We first solve the local minimum-energy problem. We then construct joint
couplings and express their excess action as one half of the integrated
squared velocity correction. The remaining proofs
estimate this difference: Stokes' formula gives a lower bound for every
coupling, and successive coordinate flows give the upper bound.

\subsection{Local velocities}
\label{subsec:proof-local}

The Poisson equations are posed on different fibres. To solve them and
differentiate their solutions in $y$, we first express all fibres in
coordinates on one fixed manifold. This change of coordinates is not
required to minimize transport cost.

\begin{proposition}[Smooth compact fibre coordinates]
\label{ass:local}
Under \Cref{ass:geometry}, there are a compact connected
manifold $M$ without boundary, a smooth
Riemannian metric $g_0$ on $M$, and a smooth bundle trivialization
$\Phi:U^+\times M\to E^+$ with $H(\Phi(y,z))=y$. Set
\[
 g_y=\Phi_y^*(g|_{M_y}),\qquad
 \Phi_y^*\mu_y=h_y\dif\mathrm{vol}_{g_0}.
\]
The metrics $g_y$ and probability densities $h_y>0$ depend smoothly on
$(y,z)\in U^+\times M$.
\end{proposition}

Here $\Phi_y(z)=\Phi(y,z)$, and $\Phi_y^*$ denotes pullback of tensors
or measures.

\begin{proof}
Let $c$ be a star centre of $U^+$, set $M=M_c$, and transport along
the flow of $N_y=(y_s-c_s)n_s+(y_t-c_t)n_t$ for time one.
The projected path is the segment from $c$ to $y$. Properness confines
its lift to a compact set, so the flow exists; reversing the segment
gives its inverse. Locally in $y$, the segments lie in one compact
subset of $U^+$. Smooth dependence of flows therefore gives a smooth
trivialization \citep[Chapter~9]{Lee2012}.
Define its positive volume Jacobian $j_y$ by
$\Phi_y^*\dif\mathrm{vol}_{M_y}=j_y\dif\mathrm{vol}_{g_0}$. Thus
$h_y=(p\circ\Phi_y)j_y/((J_H\circ\Phi_y)Z(y))$ is smooth and positive:
$Z(y)$ is the integral of a smooth positive function on the fixed compact
fibre. This argument uses volume densities and does not require orientability.
\end{proof}

On compact parameter sets, $g_y$ and $h_y$ are uniformly bounded above
and below relative to $g_0$. These coordinates define
$L^2(U;H^1(M_y))$. The coarea formula applied to \eqref{eq:occupation-measure} gives
\begin{equation*}
 \dif\mubar=\frac{p}{Z(H)}\dif\mathrm{vol}_{E},
\end{equation*}
so $\mubar$ has full support on $E$.

\begin{lemma}[Fibrewise disintegration of continuity]
\label{lem:fibre-disintegration}
Let $Z\in L^2(E,\mubar;\V)$ satisfy
$\int_E\dif\varphi(Z)\dif\mubar=0$ for every $\varphi\in C_c^1(E)$.
There is a Lebesgue-full subset of $U$ on which
\begin{equation*}
 \int_{M_y}\dif\psi(Z)\dif\mu_y=0
 \qquad\text{for every }\psi\in H^1(M_y).
\end{equation*}
\end{lemma}

\begin{proof}
In the fixed coordinates, test with
$\varphi(\Phi(y,z))=\eta(y)\psi(z)$. Verticality removes derivatives
of $\eta$, and the fundamental lemma gives the desired equality for
almost every $y$, initially for smooth $\psi$. Intersect the full-measure
sets for a countable smooth dense subset of $H^1(M,g_0)$. On this
intersection, Cauchy--Schwarz and uniform equivalence of the fibre norms
extend the equality to every $H^1$ test.
\end{proof}

\begin{proof}[Proof of \Cref{prop:local-ot}]
We first solve the Poisson equation, then verify that the resulting
velocity minimizes the energy. Write $\Phi^*n_i=e_i+\beta_i$, where
$\beta_i$ is the component tangent to the fixed fibre $M$.
In these coordinates the equation becomes
\begin{equation*}
 \int_M g_y(\nabla^{g_y}u_i,\nabla^{g_y}\psi)h_y\dif\mathrm{vol}_{g_0}
 =\int_M\psi\,\partial_i h_y\dif\mathrm{vol}_{g_0}
 -\int_M\dif\psi(\beta_i)h_y\dif\mathrm{vol}_{g_0}.
\end{equation*}
The right-hand side vanishes on constants because $\int_M h_y\dif\mathrm{vol}_{g_0}=1$;
the left-hand side depends only on gradients. We therefore solve on
functions with zero $g_0$-mean. Positivity of $h_y$ and Poincar\'e's
inequality make the bilinear form coercive, locally uniformly in $y$.
Lax--Milgram gives a unique solution
\citep[Section~6.2.1]{Evans2010}. Its smooth parameter dependence follows
by differentiating the inverse coercive operator on this fixed space.
Elliptic regularity, applied also to the differentiated equations, gives
all spatial Sobolev orders and hence joint smoothness. Subtracting the
$h_y$-mean gives the required normalization. Integration by parts in $y$
recovers \eqref{eq:poisson}; compact test support removes the boundary term.
The construction on $U^+$ has bounded derivatives over $\overline U$.

For uniqueness in the stated weak space, subtract two normalized
solutions and apply \Cref{lem:fibre-disintegration} with their
difference as test function. Its fibre gradient vanishes; connectedness
and zero mean make the difference zero.

It remains to prove minimality. The field $v_i$ is feasible by
\eqref{eq:poisson}. For another feasible $V$, their difference $Z=V-v_i$
is vertical and satisfies the homogeneous continuity constraint.
\Cref{lem:fibre-disintegration}, tested with $u_i$, gives
$\int_E g(\nabla_{M_y}u_i,Z)\dif\mubar=0$.
The normal part $n_i$ is pointwise orthogonal to $Z$, so
$\int_E g(v_i,Z)\dif\mubar=0$. Expanding
$\|v_i+Z\|_2^2$ proves the projection identity and minimality.
\end{proof}

\begin{proof}[Proof of \Cref{prop:wasserstein-speed}]
We bound the Wasserstein increment from above using the flow coupling
and from below using a test potential. Fix $y,i$ and set
$a=\|v_i\|_{L^2(\mu_y)}$. In bundle coordinates,
$\Phi^*v_i=e_i+w_i$ and the continuity equation reads
\begin{equation}
 \partial_i h_y+\diver_{g_0}(h_yw_i)=0.
 \label{eq:smooth-continuity}
\end{equation}
Here $w_i$ is the fibre component of $v_i$, and $\diver_{g_0}$ is
divergence with respect to $\mathrm{vol}_{g_0}$.
Let $Z_h=\Phi_{y+he_i}^{-1}\circ\Fl_{v_i}^h\circ\Phi_y$
be the flow in fibre coordinates, and let $J_h$ be its volume
Jacobian relative to $\mathrm{vol}_{g_0}$. The continuity equation gives
$\frac{\mathrm d}{\mathrm dh}[h_{y+he_i}(Z_h)J_h]=0$. Consequently
\begin{equation}
 (\Fl_{v_i}^{h})_\#\mu_y=\mu_{y+he_i}
 \label{eq:coordinate-flow-marginals}
\end{equation}
for either sign of $h$. The resulting coupling and the length of flow
curves give $\limsup_{h\to0}W_2/|h|\leq a$, uniformly on compact
coordinate intervals, and hence local Lipschitz continuity.

For the reverse bound, normal tubular coordinates about $M_y$ give a
compactly supported smooth potential
$\varphi(\exp_x\nu)=u_i(x)+g(n_i(x),\nu)$ near the zero section,
where $x\in M_y$, $\nu\perp T_xM_y$, and $\exp_x$ is the Riemannian
exponential map.
Its gradient on $M_y$ is $v_i$. Its Hessian is bounded, so Taylor's
formula along minimizing geodesics and Cauchy--Schwarz imply
\[
 \left|\int\varphi\dif\mu_{y+he_i}-\int\varphi\dif\mu_y\right|
 \leq aW_2(\mu_y,\mu_{y+he_i})+C W_2(\mu_y,\mu_{y+he_i})^2.
\]
Here completeness supplies minimizing geodesics; the inequality follows
first for each coupling and then for the infimum. Differentiation along
the flow gives signed difference $ha^2+O(h^2)$, whose absolute value is
$|h|a^2+O(h^2)$.
Since $W_2=O(|h|)$, division by $|h|$ gives the lower bound when $a>0$;
the case $a=0$ follows from the upper bound. Integration proves
\eqref{eq:wasserstein-dirichlet}.
\end{proof}

\subsection{Couplings, residuals, and attainment}
\label{subsec:proof-representatives}

We now pass from deterministic velocities to laws of particle surfaces.
The embedding in \eqref{eq:section-space} lets us work in ordinary Sobolev
spaces. Tubular projection and a cutoff extend maps, forms, and the
metric near $K$ to bounded ambient coefficients with bounded derivatives.
Since Sobolev functions are defined only almost everywhere, we first
choose measurable values before integrating with respect to a law $Q$.

\begin{lemma}[Product representatives and occupation kernels]
\label{lem:representatives}
The space $\mathcal S$ is a closed subset of strong $H^1(U;\R^N)$, hence
standard Borel. It admits jointly Borel representatives of its values and
weak derivatives. For each fixed bounded $C^1$ ambient map $G$ with bounded
derivative and each Borel probability $Q$ on $\mathcal S$, the Sobolev
chain rule
\begin{equation*}
 \partial_i(G\circ\sigma)(y)=\mathrm dG_{\sigma(y)}\,\partial_i\sigma(y)
\end{equation*}
holds on a $Q\otimes\dif y$-full set. The section derivatives are tangent
to $\X$ and satisfy $\mathrm dH(\partial_i\sigma)=e_i$ on such a full set.
The occupation measure induced by $Q$ is a finite Borel measure with first
marginal $\dif y$, supported on $H(x)=y$, and it has a probability kernel
$q_y$. Constraint~\eqref{eq:occupation} is equivalent to $q_y=\mu_y$
for almost every $y$.
\end{lemma}

\begin{proof}
Strong $H^1$ convergence admits an almost-everywhere convergent
subsequence. Closedness of $K$ and continuity of $H$ preserve the
constraints, so $\mathcal S$ is closed.

We next choose representatives by local averaging. For an $L^2$ function extended
by zero, averages over balls of radius $1/k$ are jointly continuous
in the function and the centre. Their limits, set to zero where they
fail to exist, are jointly Borel and represent the function almost
everywhere by Lebesgue differentiation. Apply this construction to
$\sigma$ and $\partial_i\sigma$, and fix a reference point $z_*\in M$.
Retain a value only if the averages
converge to $x\in K$ with $H(x)=y$; otherwise use $\Phi(y,z_*)$.
Retain a derivative only if its averages converge to a tangent vector
at the selected value with $\mathrm dH$-image $e_i$; otherwise use
$n_i$ at that value. All failure sets are Borel.

The Sobolev chain rule for bounded $C^1$ maps follows along almost
every coordinate line \citep{Ziemer1989}. Applied to a tubular
retraction and to $H$, it proves tangency and the section derivative
constraint. Thus these replacements change neither the Sobolev functions
nor their weak derivatives. Composition
with $G$ is continuous in $H^1$: the derivative difference splits into
a bounded coefficient times a strongly convergent derivative and a
coefficient converging in measure times a fixed $L^2$ function.
Truncation controls the latter term. The set where the chain rule fails
is therefore Borel. It has zero $y$-measure for each section, so Tonelli's
theorem gives zero $Q\otimes\dif y$-measure. Finally push
$Q\otimes\dif y$ forward by $(\sigma,y)\mapsto(y,\sigma(y))$.
Disintegration of this finite measure on standard Borel spaces gives
$q_y$, and uniqueness of disintegration gives the last assertion.
\end{proof}

The next field is the average particle velocity at a given position.
It allows us to apply the deterministic local optimality result to a
random family of sections.

\begin{lemma}[Barycentric flux]
\label{lem:barycentric-flux}
For every $Q\in\Qcoh$ and $i=s,t$ there is a unique
$b_i^Q\in L^2(E,\mubar;TE)$ such that
\begin{equation}
 \int_{\mathcal S\times U}g(a(\sigma(y)),\partial_i\sigma(y))\dif Q\dif y
 =\int_E g(a,b_i^Q)\dif\mubar
 \label{eq:barycentric-flux}
\end{equation}
for all $a\in L^2(E,\mubar;TE)$. Moreover $b_i^Q$ is feasible for
the local problem \eqref{eq:local-minimum}.
\end{lemma}

\begin{proof}
The marginal constraint and Cauchy--Schwarz bound the functional on the
left by $\|a\|_2(\int|\partial_i\sigma|^2\dif Q\dif y)^{1/2}$.
The Riesz representation theorem therefore gives $b_i^Q$.
We check its two constraints. Testing with $(\mathrm dH)^*\lambda$ for
$\lambda\in C_c^\infty(E;\R^2)$ gives
$\mathrm dH(b_i^Q)=e_i$. For $\varphi\in C_c^1(E)$, the chain rule
gives $\dif\varphi(\partial_i\sigma)=\partial_i(\varphi\circ\sigma)$.
The composition has support in the fixed compact set
$H(\supp\varphi)\Subset U$, hence lies in the zero-trace Sobolev space
$H^1_0(U)$ and its derivative
integrates to zero. Fubini proves the weak continuity constraint.
\end{proof}

To obtain one admissible coupling, we fix a reference fibre and transport
its particles along straight segments in parameter space. Other paths
need not give the same maps.

\begin{lemma}[Radial comparison without flatness]
\label{lem:radial-comparison}
Under \Cref{ass:geometry}, every star
centre $y_\star$ of $U$ admits a smooth family of diffeomorphisms
$T_y:M_{y_\star}\to M_y$, with $T_{y_\star}=\operatorname{id}$ and
$(T_y)_\#\mu_{y_\star}=\mu_y$. It has a smooth $E^+$-valued extension
near $\overline U\times M_{y_\star}$, with bounded first derivatives on
this compact set. No curvature assumption is required.
\end{lemma}

\begin{proof}
For $a=y-y_\star$, set $W_y=a_sv_s+a_tv_t$ and
$T_y=\Fl_{W_y}^1|_{M_{y_\star}}$. The projected path is
$y_\star+\tau a$. The radial hull of a sufficiently small neighbourhood
of $\overline U$ is compactly contained in $U^+$; properness confines
all lifts to a compact set. Continuation, reversal of the segment,
and smooth parameter dependence give the asserted extension and
diffeomorphisms. In fibre coordinates, \eqref{eq:smooth-continuity}
implies conservation of $h_{y_\star+\tau a}(Z_\tau)J_\tau$ along the
flow of $a_sw_s+a_tw_t$. Here $Z_\tau$ is the corresponding flow in
fibre coordinates and $J_\tau$ its volume Jacobian relative to
$\mathrm{vol}_{g_0}$. Change of variables gives the marginals.
\end{proof}

\begin{proof}[Proof of \Cref{prop:residual}]
The smooth family from \Cref{lem:radial-comparison} defines a law of sections by pushing
$\mu_{y_\star}$ forward under $z\mapsto[T(\cdot,z)]$. This map is
continuous in $H^1$, and bounded derivatives give finite action.
For any admissible $Q$, the averaged velocity $b_i^Q$ is feasible for
the local problem. Hence local optimality gives the cancellation
\[
 \int g(v_i(\sigma),r_i)\dif Q\dif y
 =\int_E g(v_i,b_i^Q-v_i)\dif\mubar=0,
\]
where the first equality uses \eqref{eq:barycentric-flux} and the
marginal constraint. The chain rule makes $r_i$ vertical. Expanding
$|\partial_i\sigma|^2=|v_i(\sigma)+r_i|^2$, the mixed term vanishes
by this equality and the remaining terms give the residual identity.
The finite-action coupling constructed above makes the gap finite.
\end{proof}

\begin{proof}[Proof of \Cref{prop:coherent-attainment}]
We use the direct method for the infimum in \eqref{eq:gap}: compactness
gives a limiting law, and we then
check that its marginals are unchanged.

\textit{Step 1: compact energy sublevels.}
Regard sections as elements of $\mathcal L=L^2(U;\R^N)$, and extend
their Dirichlet energy to $+\infty$ outside $\mathcal S$.
This functional is lower semicontinuous with compact sublevels.
Indeed, compactness of $K$ bounds values, and uniform metric
equivalence bounds derivatives of a bounded-energy sequence.
Rellich compactness gives strong $L^2$ convergence along a subsequence;
an almost-everywhere convergent subsequence preserves the section
constraints. For an ambient
positive-definite extension $\mathbf G$ of $g$, boundedness and
convergence in measure give
$\mathbf G(u_n)^{1/2}\partial_i u_n\rightharpoonup
\mathbf G(u)^{1/2}\partial_i u$ in $L^2$: test against a fixed $L^2$
field, transfer the matrix to the test field, and use its strong $L^2$
convergence. Weak lower semicontinuity of the squared norm proves the claim.

\textit{Step 2: convergence of the laws.}
For a minimizing sequence of laws, Markov's inequality on these compact
sublevels gives tightness in $\mathcal L$. Prokhorov's theorem gives
a narrowly convergent subsequence. Portmanteau's theorem applied to the
extended energy shows that the limiting law has no larger action and
is concentrated on $\mathcal S$.
The continuous injection of the closed $H^1$ space $\mathcal S$ into
$\mathcal L$ has a Borel inverse on its image by Lusin--Souslin
\citep[Section~15]{Kechris1995}. Thus the limit defines the required
Borel law on $\mathcal S$.

\textit{Step 3: preservation of marginals.}
Point evaluation is not continuous in $L^2$, so we instead use the bounded
continuous functionals $u\mapsto\int_U\eta(y)\varphi(u(y))\dif y$, with
$\eta\in C_c(U)$ and $\varphi\in C_b(\R^N)$.
Here $C_c$ and $C_b$ denote continuous functions with compact support
and bounded continuous functions, respectively. Their convergence follows
from convergence in measure. Finite sums of these products are dense
in $C_0(U\times K)$, the continuous functions vanishing at infinity,
so the limiting occupation measure is the prescribed
one. The limit is admissible and attains the infimum.
\end{proof}

\subsection{Curvature and flatness}
\label{subsec:proof-stokes}

For a smooth particle surface, mixed parameter derivatives commute.
\Cref{lem:weak-stokes} expresses the resulting integration-by-parts identity
in a form that also holds for Sobolev surfaces. In the statement,
$T^*$ denotes a cotangent bundle, $\Lambda^2T^*U$ the bundle of
alternating two-forms on $U$, and $\mathcal D'(U)$ the space of distributions.

\begin{lemma}[Weak Stokes for Sobolev sections]
\label{lem:weak-stokes}
Let $\sigma\in\mathcal S$ and let
$\theta\in C_c^1(E^+;T^*E^+)$ satisfy
$H(\operatorname{supp}\theta)\Subset U$. Then
$\sigma^*\theta\in L^2(U;T^*U)$,
$\sigma^*(\dif\theta)\in L^1(U;\Lambda^2T^*U)$, and
\begin{equation*}
 \dif(\sigma^*\theta)=\sigma^*(\dif\theta)
       \quad\text{in }\mathcal D'(U),\qquad
 \int_U\sigma^*(\dif\theta)(e_s,e_t)\dif y=0.
\end{equation*}
\end{lemma}

\begin{proof}
For smooth maps the pullback identity follows by differentiating the
coefficients of the form; the mixed second derivatives cancel.
To pass to Sobolev maps, extend $\theta$ to a compactly supported
ambient $C^1$ form.
The Lipschitz extension theorem and mollification approximate $\sigma$
strongly in $H^1$ by smooth ambient maps $u_k$ \citep{AdamsFournier2003}.
The one-form pullbacks converge in $L^2$: coefficients are bounded
and converge in measure, while derivatives converge strongly.
For the two-form, each derivative product converges in $L^1$ by
Cauchy--Schwarz. Multiplication by the bounded convergent coefficient
preserves this limit, by truncating the fixed $L^1$ product.
The smooth pullback identity thus passes to distributions. Tangency
identifies the limit with the intrinsic pullback. Both pullbacks vanish
outside $H(\supp\theta)$. Testing the distributional identity against
a cutoff equal to one near this compact set proves the zero integral.
\end{proof}

Write $\dif_\V\alpha=\dif\widetilde\alpha|_{\V\times\V}$.

\begin{proposition}[Residual--curvature identity]
\label{prop:residual-curvature}
Under \Cref{ass:geometry}, every
$Q\in\Qcoh$ and $\alpha\in C_c^1(E;\V^*)$ satisfy
\begin{equation}
 \int_E\alpha(\Omega)\dif\mubar
 =\int_{\mathcal S\times U}\left[
 \dif\widetilde\alpha(v_s,r_t)-\dif\widetilde\alpha(v_t,r_s)
 +(\dif_\V\alpha)(r_s,r_t)\right]\dif Q\dif y.
 \label{eq:residual-curvature}
\end{equation}
All forms on the right are evaluated at $\sigma(y)$.
\end{proposition}

\begin{proof}
We apply \Cref{lem:weak-stokes} after separating each velocity into its local
minimizer and its residual from \eqref{eq:residual}.
Since $\widetilde\alpha(v_i)=0$, the exterior derivative formula gives
$\dif\widetilde\alpha(v_s,v_t)=-\alpha(\Omega)$.
Expand $\dif\widetilde\alpha(\partial_s\sigma,\partial_t\sigma)$
using $\partial_i\sigma=v_i+r_i$. Its integral is zero by weak Stokes.
Average over $Q$ and use occupation on the curvature term. Fubini
applies: the coefficients are bounded, the linear residual terms are
integrable by Cauchy--Schwarz, and the quadratic term belongs to $L^1$.
\end{proof}

\begin{proof}[Proof of \Cref{thm:certificate}]
Fix an admissible law $Q$. Using the coefficients in \eqref{eq:witness-terms},
we estimate the linear and quadratic residual terms separately. Let $z$ be the $L^2(Q\otimes\dif y)$ norm of
$(r_s,r_t)$; the residual identity \eqref{eq:residual-identity} gives $z^2=2(\Aact(Q)-\Eloc)$.
Cauchy--Schwarz bounds the linear terms in the identity \eqref{eq:residual-curvature} by
$B_\alpha z$, and $2|r_s||r_t|\leq|r_s|^2+|r_t|^2$ bounds the
quadratic term by $A_\alpha z^2$. Consequently,
\[
 I_\alpha\leq B_\alpha z+A_\alpha z^2.
\]
For $I_\alpha>0$, feasibility implies $A_\alpha+B_\alpha>0$.
The right-hand side is increasing for $z\geq0$, so solving the quadratic
inequality gives
$z\geq2I_\alpha/(B_\alpha+\sqrt{B_\alpha^2+4A_\alpha I_\alpha})$.
This form of the bound remains valid when $A_\alpha=0$. Squaring and
dividing by two gives the action bound for $Q$. Taking the infimum over
laws and then the supremum over test forms proves the theorem.
\end{proof}

\begin{proof}[Proof of \Cref{thm:flatness}]
\textit{Step 1: nonzero curvature gives positive excess.}
If $\Omega(x_0)\ne0$, choose a nonnegative compactly supported cutoff
$\chi$ positive at $x_0$, and let $\alpha=\chi g(\Omega,\cdot)|_\V$.
Full support of $\mubar$ gives $I_\alpha>0$; hence the gap is positive.

\textit{Step 2: zero curvature gives an optimal realization.}
Assume $\Omega=0$ and use the radial family $F=T$ from \Cref{lem:radial-comparison}.
We must check that this family has both prescribed derivatives.
In the following coordinate calculation, $\mathrm dV$ denotes the
Jacobian of the coordinate components of a vector field $V$.
For $X_\tau=\Fl_{W_y}^{\tau}(z)$, the parameter derivative
$Y_i=\partial_{y_i}X_\tau$ solves
$\dot Y_i=\mathrm dW_y(X_\tau)Y_i+v_i(X_\tau)$, $Y_i(0)=0$.
Since $[W_y,v_i]=0$, the vector $\tau v_i(X_\tau)$ solves the same
equation. Thus $\partial_iF=v_i(F)$ and its law has zero residual.

\textit{Step 3: uniqueness.}
We show that every zero-residual section keeps the same initial label.
Let $\Lambda(x)$ be the initial label given by the
inverse of $F$. The radial extension makes $\Lambda$ smooth near $K$
with bounded derivative, and $\mathrm d\Lambda(v_i)=0$ on $E$.
An optimal law has zero residual. The Sobolev chain rule, after embedding
the label fibre, gives $\partial_i(\Lambda\circ\sigma)=0$. Connectedness
of $U$ makes this label constant almost everywhere, so its sections are
$[F(\cdot,z)]$. The continuous injection $z\mapsto[F(\cdot,z)]$ from
the compact label fibre has continuous inverse on its image. Occupation,
tested with functions of $\Lambda$, determines the label law as
$\mu_{y_\star}$, hence determines $Q_F$. Finally,
$F(\gamma(\tau),z)$ solves the horizontal ordinary differential equation
(ODE) along every piecewise
$C^1$ path. ODE uniqueness proves the endpoint formula and uniqueness
of the normalized smooth family.
\end{proof}

\subsection{The small-square action bounds}
\label{sec:quantitative-proofs}

\begin{proof}[Proof of \Cref{thm:quartic-action}]
For the upper bound we estimate the explicit coupling $Q_r$.
For the lower bound we use a test form supported on $U_r$, so the
estimate applies to every admissible coupling.

\textit{Upper bound.}
Choose $r_*>0$ with $C_*=y_0+[-2r_*,2r_*]^2\subset U$.
All coordinate flows exist on the compact inverse image of $C_*$.
\Cref{eq:coordinate-flow-marginals} gives exact marginals
for $F_r$. Smooth parameter dependence gives bounded derivatives up
to $\overline{U_r}$, so $Q_r$ is admissible. Set
$L=\sup_{H^{-1}(C_*)}\|\nabla v_t\|_{\mathrm{op}}$, where the operator
norm uses $g$. The variational
equation and Gronwall give $\|\mathrm d\Fl_{v_t}^{a}\|\leq e^{L|a|}$.

The last flow is in the $t$ direction, so $\partial_tF_r=v_t(F_r)$
and its $t$ residual is zero. Put $\tau=t-t_0$ and
$a=\Fl_{v_s}^{s-s_0}(z)$. Its $s$ residual is
\[
 R_s=\mathrm d(\Fl_{v_t}^{\tau})_a v_s(a)
       -v_s(\Fl_{v_t}^{\tau}(a)).
\]
Taking the covariant derivative $\nabla_\tau$ along
$\tau\mapsto\Fl_{v_t}^\tau(a)$ gives
$\nabla_\tau R_s=\nabla_{R_s}v_t+\Omega(\Fl_{v_t}^\tau(a))$,
with $R_s(0)=0$. Variation of constants yields
\begin{equation}
 R_s(s,t,z)=\int_0^\tau
 \mathrm d(\Fl_{v_t}^{\tau-u})_{\Fl_{v_t}^{u}(a)}
 \Omega(\Fl_{v_t}^{u}(a))\dif u.
 \label{eq:axial-residual-integral}
\end{equation}
Cauchy--Schwarz on the signed segment, followed by marginal preservation,
therefore gives
$\int|R_s(s,t,z)|^2\dif\mu_{y_0}(z)
\leq e^{2Lr}(t-t_0)^2\sup_{U_r}\kappa^2$.
The residual identity \eqref{eq:residual-identity} and
$\frac12\int_{-r}^r\int_{-r}^r t^2\dif s\dif t=\frac23r^4$
prove \eqref{eq:quartic-upper}. Thus the two factors of $r^2$ come from
the squared velocity correction and the area of the square.

For comparison with the numerical results, the leading coefficient of
this particular construction can also be read from
\eqref{eq:axial-residual-integral}. Write $F_y(z)=F_r(y,z)$ and identify
vertical tangent spaces with $T_zM_{y_0}$ using the fibre diffeomorphism
$F_y$. Smoothness and compactness give, uniformly for $y\in U_r$ and
$z\in M_{y_0}$,
\[
 (\mathrm d_zF_y)^{-1}R_s
   =(t-t_0)\Omega(z)+O(r|t-t_0|),
 \qquad F_y^*g=g|_{M_{y_0}}+O(r).
\]
The oriented integral in \eqref{eq:axial-residual-integral} gives this
expansion for both signs of $t-t_0$. Hence
$|R_s|_g^2=(t-t_0)^2|\Omega(z)|_g^2+O(r(t-t_0)^2)$, and integration
in the residual identity \eqref{eq:residual-identity} yields
\begin{equation}
 \Aact_{U_r}(Q_r)-\Eloc(U_r)
   =\frac23\kappa(y_0)^2r^4+O(r^5).
 \label{eq:axial-leading-action}
\end{equation}
Equation \eqref{eq:axial-leading-action} identifies the leading coefficient
of the ordered construction.

\textit{Lower bound.}
Suppose $\kappa_0=\kappa(y_0)>0$. By continuity, reduce $r_*$ so that
$\kappa_0^2/2\leq\kappa(y)^2\leq2\kappa_0^2$ on $\overline{U_{r_*}}$.
Choose nonnegative, nonzero $\chi\in C_c^\infty((-1,1)^2)$, set
$\chi_r(y)=\chi((y-y_0)/r)$, and test with
$\alpha_r=(\chi_r\circ H)\beta$, where $\beta=g(\Omega,\cdot)|_\V$.
Properness gives compact support. Let $\widetilde\beta$ be the extension
of $\beta$ that annihilates $v_s,v_t$, as for $\widetilde\alpha$ in
\Cref{subsec:curvature-certificate}. The product rule gives, for vertical $R$,
\[
 \dif\widetilde\alpha_r(v_i,R)
 = (\partial_i\chi_r)\beta(R)+\chi_r\dif\widetilde\beta(v_i,R),
 \qquad \dif_\V\alpha_r=\chi_r\dif_\V\beta.
\]
Write $I_r=I_{\alpha_r}$, $B_r=B_{\alpha_r}$, and $A_r=A_{\alpha_r}$,
with the quantities in \eqref{eq:witness-terms} computed on $U_r$.
Consequently there are fixed positive $c_1,C_1$ and finite $C_2,C_3$ with
\[
 I_r\geq c_1\kappa_0^2r^2,\qquad
 B_r\leq C_1\kappa_0+C_2r,\qquad A_r\leq C_3.
\]
Indeed, $I_r=\int_{U_r}\chi_r(y)\kappa(y)^2\dif y$, which gives
the first bound by $y=y_0+r\xi$. In $B_r$, a derivative of $\chi_r$
contributes $r^{-1}$, cancelled by the $r$ from the $L^2$ area factor.
The terms without that derivative are $O(r)$, since the coefficients
of $\dif\widetilde\beta$ are bounded on the compact tube.
Vertical differentiation does not act on $\chi_r\circ H$, giving the
bound for $A_r$. Choose
$r_0\leq\min(r_*,1)$ with $C_2r_0\leq\kappa_0$.
For any admissible $Q$, put $G=\Aact_{U_r}(Q)-\Eloc(U_r)$ and
$z=\sqrt{2G}/(\kappa_0r^2)$. Substituting these estimates into the
curvature bound \eqref{eq:certificate} and dividing by $\kappa_0^2r^2$ gives
$c_1\leq(C_1+1)z+C_3r^2z^2\leq(C_1+1)z+C_3z^2$.
The increasing polynomial forces $z\geq c_2>0$, uniformly in $r,Q$.
This proves the lower bound. The upper bound \eqref{eq:quartic-upper} and continuity
of $\kappa$ prove \eqref{eq:quartic-sandwich}, with constants depending
only on the stated local data.
\end{proof}

\subsection{Ordered flows and the torus calculation}
\label{subsec:proof-order}

\begin{proof}[Proof of \Cref{thm:order-defect}]
We expand the two compositions in the same coordinates. Their
first-order terms agree; the mixed second-order terms give the bracket.
Properness and flow continuation on a compact square about $y_0$ give
both ordered diffeomorphisms and their inverses. Each segment preserves
marginals by \eqref{eq:coordinate-flow-marginals}, proving
\eqref{eq:ordered-marginals}. In bundle coordinates, Taylor expansion gives
\[
 P_{st}^{\varepsilon}(x)
 =x+\varepsilon(v_s+v_t)(x)
 +\varepsilon^2\left(\tfrac12\mathrm dv_s\,v_s
       +\mathrm dv_t\,v_s+\tfrac12\mathrm dv_t\,v_t\right)(x)
 +O(|\varepsilon|^3).
\]
For $P_{ts}^{\varepsilon}$ only the mixed term changes, from
$\mathrm dv_t\,v_s$ to $\mathrm dv_s\,v_t$. A finite chart cover of
the compact starting fibre makes the remainder uniform. Subtracting,
\[
 P_{st}^{\varepsilon}(x)-P_{ts}^{\varepsilon}(x)
 =\varepsilon^2(\mathrm dv_t\,v_s-\mathrm dv_s\,v_t)(x)
   +O(|\varepsilon|^3)
 =\varepsilon^2\Omega(x)+O(|\varepsilon|^3).
\]
Both endpoints are $O(|\varepsilon|)$ from $x$ and their separation is
$O(\varepsilon^2)$. The straight coordinate segment gives a path of
this length. A minimizing short geodesic remains in the same chart;
smoothness compares its metric with $g_x$ with relative error
$O(|\varepsilon|)$. Thus distance differs from the $g_x$-norm of the
coordinate separation by $O(|\varepsilon|^3)$, uniformly on the fibre.
Taking $L^2(\mu_{y_0})$ norms and using the reverse triangle inequality
proves the expansion. For intrinsic distance, use the common endpoint
fibre in bundle coordinates. Its metrics vary smoothly on a compact
family, so the same short-geodesic comparison applies.
\end{proof}

\begin{proof}[Proof of \Cref{cor:squared-defect}]
Write $\delta_{\mathrm{loop}}=c\varepsilon^2+R_\varepsilon$, with
$c=\|\Omega\|_{L^2(\mu_{y_0})}$ and
$|R_\varepsilon|\leq C|\varepsilon|^3$. Squaring gives a remainder
$2c\varepsilon^2R_\varepsilon+R_\varepsilon^2=O(|\varepsilon|^5)$,
also when $c=0$.
\end{proof}

\begin{proof}[Proof of \Cref{prop:torus-example}]
We first realize the proposed conditional densities in the coarea
setting, then calculate their local velocities and curvature.
Choose $\delta<\delta_+<\delta_1$ with $\delta_1(|a|+|b|)<1$,
and a smooth cutoff $0\leq\chi\leq1$ supported in $(-\delta_1,\delta_1)^2$
and equal to one on $U^+=(-\delta_+,\delta_+)^2$. For any smooth
positive probability density $p_0$ on $\R^2$, set
\[
 p(y,x)=\frac{p_0(y)}{(2\pi)^2}
 [1+\chi(y)(as\cos x_1+bt\cos(x_1+x_2))].
\]
This is a smooth positive probability density for product volume.
The projection $H$ is proper, $J_H=1$, and $Z=p_0$, giving the
required conditionals by \eqref{eq:coarea}. The product manifold satisfies the geometric
assumptions.

The normal coordinate fields are constant on the product manifold.
In this calculation, $\nabla$, $\diver$, and
$\Delta=\partial_{x_1}^2+\partial_{x_2}^2$ act in the flat fibre
coordinates $x$.
The fibre potentials solve $-\diver(\rho_y\nabla u_i)=\partial_i\rho_y$.
At the origin, $u_s=a\cos x_1$ and $u_t=(b/2)\cos(x_1+x_2)$.
To compute curvature, we also need the parameter derivatives of these
potentials. For $q=\partial_su_t-\partial_tu_s$ at the origin,
differentiating the two Poisson equations and subtracting gives
$\Delta q=(ab/4)(\cos x_2-\cos(2x_1+x_2))$, so, up to a constant,
$q=ab[-\cos x_2/4+\cos(2x_1+x_2)/20]$.
Substitution into
$\Omega=\nabla q+\mathrm dw_t\,w_s-\mathrm dw_s\,w_t$
gives \eqref{eq:torus-curvature}. Fourier orthogonality yields
$\|\Omega\|_2^2/(a^2b^2)=1/8+1/200+1/50=3/20$.
\Cref{thm:flatness} gives a positive gap, and \Cref{prop:residual}
makes it finite.
\Cref{thm:order-defect} gives \eqref{eq:torus-defect}.
\end{proof}

\section{Conclusions}
\label{sec:conclusions}
For a fixed two-parameter family, realizing both local OT velocities by
one particle law imposes a compatibility constraint. We quantify its
energy cost by a curvature estimate valid for every admissible Sobolev
law. At nonzero-curvature centres, the lower bound and an explicit
ordered construction give the sharp quartic order of the minimum excess
on small squares. Vanishing curvature characterizes zero gap and a
unique minimizing law under the stated regular coarea assumptions.

The fitted two-phase experiment illustrates how changing the coupling
can reduce action while preserving the conditional laws. A natural
next step is to refine the quartic bounds by identifying the optimal
leading coefficient and a construction attaining it.
Further questions concern regularity of minimizing
laws in the nonflat case and extensions beyond smooth positive densities
and compact regular fibres.

\section*{Declarations}
\paragraph{Data source.}
Data are from the UCI \emph{Individual Household Electric Power Consumption}
dataset \citep{HebrailBerard2006}.

\paragraph{AI assistance.}
OpenAI Codex was used to improve the language and presentation of the
manuscript.

\bibliographystyle{cmsplainnat}
\bibliography{references}
\end{document}